\documentclass[11pt,reqno]{amsart}

\usepackage[T1]{fontenc}
\usepackage{lmodern}
\usepackage{microtype}
\usepackage[margin=1.12in]{geometry}
\usepackage{amsmath,amssymb,amsthm,mathtools}
\usepackage{enumitem}
\usepackage{booktabs,array}
\usepackage{xcolor}
\usepackage[colorlinks=true,linkcolor=blue,citecolor=blue,urlcolor=blue]{hyperref}

\numberwithin{equation}{section}

\theoremstyle{plain}
\newtheorem{theorem}{Theorem}[section]
\newtheorem{proposition}[theorem]{Proposition}
\newtheorem{lemma}[theorem]{Lemma}
\newtheorem{corollary}[theorem]{Corollary}

\newcommand{\N}{\mathbb N}
\newcommand{\Pp}{\mathbb P}
\newcommand{\E}{\mathbb E}
\newcommand{\R}{\mathbb R}
\newcommand{\QQ}{\mathbb Q}
\newcommand{\TV}{\mathrm{TV}}
\newcommand{\PD}{\operatorname{PD}}
\newcommand{\Gam}{\operatorname{Gamma}}
\newcommand{\Exp}{\operatorname{Exp}}
\newcommand{\Unif}{\operatorname{Unif}}
\newcommand{\Law}{\mathcal L}
\newcommand{\one}{\mathbf 1}
\newcommand{\W}{\mathbf W}
\newcommand{\bP}{\mathbf P}
\DeclareMathOperator{\Li}{Li}

\title[Two phase transitions in Billingsley's model]
{Tilting Billingsley's model toward a giant prime:\\ two phase transitions}

\author{Wen Sun}
\address{School of Mathematical Sciences, University of Science and Technology of China,
96 Jinzhai Road, Hefei 230026, China}
\email{wensun.ustc@gmail.com}


\begin{document}

\begin{abstract}
Billingsley's theorem~\cite{Billingsley1972} states that the normalized
logarithms of the prime factors of a uniform random integer in $[1,x]$
converge to the
Poisson--Dirichlet law $\PD(1)$, while weighting an integer $n$ by the
generalized divisor function $d_\theta(n)$ gives $\PD(\theta)$
\cite{ElboimGorodetsky2024}.  We ask what
remains of this picture when the integer is also rewarded according to its
largest prime factor $P^+(n)$.  For fixed $\theta,\beta>0$ and
$\gamma\ge0$, we sample $N_x=n\le x$ with probability proportional to
$d_\theta(n)\exp\{\beta H_\gamma(n)\}$, where
$H_\gamma(n)=(\log P^+(n))^\gamma$ for $\gamma>0$, while
$H_0(1)=0$ and $H_0(n)=1$ for $n\ge2$.  Two phase transitions occur.  At $\gamma=0$
the $\PD(\theta)$ partition survives, whereas every fixed $\gamma>0$
forces one prime to carry asymptotically all logarithmic mass.  The second
transition, at $\gamma=1$, concerns the cofactor
$R_x=N_x/P^+(N_x)$.  With $a_x=\beta\gamma(\log x)^{\gamma-1}$, its law is
asymptotic in total variation to
$\QQ_{a_x}(m)=d_\theta(m)m^{-1-a_x}/\zeta(1+a_x)^\theta$.  Thus, for
$0<\gamma<1$, $a_x\log R_x$ has a $\Gam(\theta,1)$ limit, with shape
$\theta$ and rate one, and the normalized logarithmic prime factors of
$R_x$ have an independent $\PD(\theta)$ limit; at
$\gamma=1$, $R_x$ has a nondegenerate discrete limit; and for $\gamma>1$,
$R_x=1$ with high probability.  We also determine the joint limits
involving $N_x/x$ and the normalizing constants, which include the classical
Alladi--Erd\H{o}s asymptotic~\cite{AlladiErdos1977} as a special case.
\end{abstract}

\maketitle

\section{Introduction}\label{sec:intro}

\subsection{Billingsley's theorem}\label{sec:billingsley}

Let $J_x$ be uniformly distributed on the integers in $[1,x]$, and list
its prime factors with multiplicity as
\[
 P_1(J_x)\ge P_2(J_x)\ge\cdots,
\]
the list being completed by ones; on $\{J_x=1\}$, an event of probability
$O(x^{-1})$, all entries equal one.  In a standard equivalent formulation,
Billingsley \cite{Billingsley1972} proved that
\begin{equation}\label{eq:billingsley}
 \left(\frac{\log P_i(J_x)}{\log x}\right)_{i\ge1}
 \Longrightarrow\PD(1),
\end{equation}
where $\PD(\theta)$ denotes the Poisson--Dirichlet distribution with
parameter $\theta>0$, a law on decreasing nonnegative sequences with sum
one.  The family was introduced by Kingman \cite{Kingman1975}, and its
gamma subordinator construction is recalled in
Section~\ref{sec:euler}.  Since
$\log J_x/\log x\to1$ in probability, the same limit holds with
$\log J_x$ in the denominator.  Donnelly and Grimmett
\cite{DonnellyGrimmett1993} gave a direct probabilistic proof of this
version: when listed by sampling in size biased order, the normalized
logarithmic prime factors converge to the Griffiths--Engen--McCloskey
$\operatorname{GEM}(1)$ law, whose decreasing rearrangement has law
$\PD(1)$.  On the permutation side, Watterson's classical Ewens limit
states that, for every $\theta>0$, the ranked cycle lengths of an
Ewens$(\theta)$ random permutation, divided by the permutation size,
converge to $\PD(\theta)$ \cite{Watterson1976}; see also
\cite{ABT2003}.  In this analogy, Billingsley's theorem
\eqref{eq:billingsley} is the arithmetic counterpart of the special case
$\theta=1$.

On the arithmetic side, an analogous parameter is carried by the
generalized divisor function $d_\theta$, the multiplicative function
determined by
\begin{equation}\label{eq:dtheta-euler}
 \sum_{n\ge1}\frac{d_\theta(n)}{n^{s}}
 =\prod_p(1-p^{-s})^{-\theta}
 =\zeta(s)^\theta,
 \qquad \operatorname{Re}s>1,
\end{equation}
where $\zeta$ is the Riemann zeta function.  Throughout the paper, every
sum or product indexed by $p$ is over primes.  Elboim and Gorodetsky
\cite[Theorem~1.1]{ElboimGorodetsky2024} obtained the corresponding
arithmetic extension for every $\theta>0$ as part of their theorem for
multiplicative weights: if $n\le x$ is sampled with probability
proportional to $d_\theta(n)$, then its normalized logarithmic prime
factors converge to $\PD(\theta)$.  This cutoff law, weighted by
$d_\theta$, is the arithmetic analogue of the Ewens measure and will
serve as our reference law.

\subsection{The model}\label{sec:model}

We perturb this cutoff law by a reward that sees only
the largest prime factor.  Throughout the paper
\[
 \theta>0,\qquad \beta>0,\qquad \gamma\ge0
\]
are fixed, and all limits are taken as $x\to\infty$.

We first describe the values of $d_\theta$.  If $n=\prod_pp^{\nu_p(n)}$ is
the prime factorization of $n$, where $\nu_p(n)$ is the exponent of $p$ in
$n$, then
\begin{equation}\label{eq:dtheta-local}
 d_\theta(n)
 =\prod_{p:\,\nu_p(n)>0}\frac{(\theta)_{\nu_p(n)}}{\nu_p(n)!},
 \qquad d_\theta(1)=1,
\end{equation}
where $(\theta)_0=1$ and
$(\theta)_k=\theta(\theta+1)\cdots(\theta+k-1)=\Gamma(\theta+k)/\Gamma(\theta)$
for $k\ge1$, $\Gamma$ being Euler's gamma function; expanding one Euler
factor at a time in \eqref{eq:dtheta-euler} gives \eqref{eq:dtheta-local}.
Thus $d_\theta$ is multiplicative, the value $d_\theta(p^k)=(\theta)_k/k!$
does not depend on the prime $p$, $d_\theta(p)=\theta$, and
$d_1\equiv1$.  Hence, for $\theta=1$, the unperturbed reference law is
uniform on the integers up to $x$.

For $n\ge2$ let $P^+(n)$ be the largest prime factor of $n$, and put
$P^+(1)=1$.  The rewarded statistic is
\[
 H_\gamma(n)=
 \begin{cases}
  (\log P^+(n))^\gamma,&\gamma>0,\\
  \one_{\{n\ge2\}},&\gamma=0,
 \end{cases}
\]
where $\one_B$ denotes the indicator of an event $B$.  The second line
removes the ambiguity $0^0$: for every $n\ge2$ it is the pointwise limit
of the first line as $\gamma\downarrow0$, and $H_\gamma(1)=0$ for every
$\gamma\ge0$.  For $x\ge3$, the model is the probability distribution
\begin{equation}\label{eq:model}
 \Pp_x(N_x=n)
 =\frac{1}{Z_x}\,d_\theta(n)\exp\{\beta H_\gamma(n)\},
 \qquad 1\le n\le x,
\end{equation}
with normalizing constant
\begin{equation}\label{eq:partition}
 Z_x=\sum_{n\le x}d_\theta(n)\exp\{\beta H_\gamma(n)\}.
\end{equation}
In the language of statistical mechanics, \eqref{eq:model} is a Gibbs
measure whose reference law is the cutoff law weighted by $d_\theta$ and whose
energy is $-\beta H_\gamma$.  For $\gamma>0$, set
$V_\gamma(s)=\beta s^\gamma$ for $s\ge0$.  Then
$\beta H_\gamma(n)=V_\gamma(\log P^+(n))$, and we call $V_\gamma$ the
external field.  The parameter $\theta$ controls the reference weight,
$\beta$ is the strength of the field, $\gamma$ determines how fast the
field grows with the largest prime factor, and the cutoff $x$ plays the
role of the volume.  We suppress $\theta,\beta,\gamma$ from the notation and
write $\E_x$ for expectation under $\Pp_x$.  When $\theta=1$ and
$\gamma=1$, the normalizing constant is the power sum
$Z_x=\sum_{n\le x}P^+(n)^\beta$.

\subsection{Overview of the results}\label{sec:overview}

We ask how a reward depending only on the largest prime factor changes the
Poisson--Dirichlet partition, and what law governs the remaining integer once
a giant prime has formed.  A small positive $\gamma$ might be expected to
give only a mild deformation.  Instead it changes the macroscopic
factorization completely, while the entire cofactor admits a simple Euler
product approximation in total variation.  Table~\ref{tab:phases}
summarizes the answer; precise statements appear in
Section~\ref{sec:results}.  For
$n\ge2$, let $P_1(n)\ge P_2(n)\ge\cdots$ be its prime factors listed with
multiplicity and completed by ones, and set
$\W(n)=(\log P_i(n)/\log n)_{i\ge1}$, with
$\W(1)=(1,0,0,\ldots)$.  This separate convention is needed because
$\log1=0$, and it keeps $\W(n)$ in the unit simplex.  We also write
$R_x=N_x/P^+(N_x)$, $L=\log x$,
$a_x=\beta\gamma L^{\gamma-1}$, and $b_x=1+a_x$.  We use
$G\sim\Gam(\theta,1)$ with shape $\theta$ and rate one,
$U\sim\Unif(0,1)$, $E\sim\Exp(1)$ with rate one, and
$\Pp(B_\beta\le u)=u^{1+\beta}$ for $0\le u\le1$.  The Euler product
law $\QQ_a$ is defined in \eqref{eq:Qa-intro} below, and
$M_\beta$ denotes a random integer with law $\QQ_\beta$.  We call $N_x/x$
the cutoff location of the sampled integer.  When $\gamma>1$ this ratio
converges to one, and its distance from the upper cutoff is resolved on
the finer scale $b_x\log(x/N_x)$.

\begin{table}[t]
\centering
\small
\renewcommand{\arraystretch}{1.45}
\setlength{\tabcolsep}{4.5pt}
\begin{tabular}{@{}lcccc@{}}
\toprule
 & $\gamma=0$ & $0<\gamma<1$ & $\gamma=1$ & $\gamma>1$\\
\midrule
$\W(N_x)$
 & $\PD(\theta)$
 & $(1,0,\ldots)$
 & $(1,0,\ldots)$
 & $(1,0,\ldots)$\\
cofactor $R_x$
 & $R_x\xrightarrow{\Pp}\infty$
 & $a_x\log R_x\Rightarrow G$
 & $R_x\Rightarrow M_\beta$
 & $R_x=1$ w.h.p.\\
$\W(R_x)$
 & ---
 & $\PD(\theta)$
 & ---
 & ---\\
cutoff location
 & $N_x/x\Rightarrow U$
 & $N_x/x\Rightarrow U$
 & $N_x/x\Rightarrow B_\beta$
 & $b_x\log(x/N_x)\Rightarrow E$\\
$\Pp_x(N_x\text{ prime})$
 & $\sim\Gamma(\theta+1)L^{-\theta}$
 & $\sim a_x^{\theta}$
 & $\to\zeta(1+\beta)^{-\theta}$
 & $\to1$\\
$Z_x$ up to constants
 & $xL^{\theta-1}$
 & $xe^{\beta L^\gamma}L^{\theta(1-\gamma)-1}$
 & $x^{1+\beta}/L$
 & $xe^{\beta L^\gamma}L^{-\gamma}$\\
\bottomrule
\end{tabular}
\medskip
\caption{Phase diagram for \eqref{eq:model}.  The $\W(N_x)$ row gives
weak convergence at $\gamma=0$ and convergence in probability otherwise;
``w.h.p.'' means with probability tending to one, and ``---'' means that
the table does not display a limit for that entry.}
\label{tab:phases}
\end{table}

\emph{The first transition, at $\gamma=0$.}  This transition concerns the
macroscopic normalized prime factor partition.  At $\gamma=0$ the reward is
constant away from the integer one, so the $\PD(\theta)$ limit survives and
is asymptotically independent of the uniform limit of $N_x/x$
(Theorem~\ref{thm:zero}).  For every fixed $\gamma>0$, however small, one
prime instead carries asymptotically all logarithmic mass; see
Corollary~\ref{cor:giant-prime}.

\emph{The second transition, at $\gamma=1$.}  This transition concerns the
scale of the cofactor $R_x$ after the giant prime has formed.  Our central
result (Theorem~\ref{thm:cofactor}) identifies the entire cofactor law,
not only one of its statistics.  For every fixed $\gamma>0$, it is
asymptotic in total variation to the following Euler product law on
$\N=\{1,2,3,\ldots\}$.  For $a>0$, define
\begin{equation}\label{eq:Qa-intro}
 \QQ_{a}(m)=\frac{d_\theta(m)\,m^{-1-a}}{\zeta(1+a)^\theta},
 \qquad m\in\N,
 \qquad a>0.
\end{equation}
By \eqref{eq:dtheta-euler}, $\QQ_a$ is the probability law induced by the
Euler product for $\zeta(1+a)^\theta$.  The cofactor approximation uses this family at
$a=a_x=\beta\gamma(\log x)^{\gamma-1}$.
The parameter $a_x$ is the slope of the reward at $\log x$.  It tends to
$0$, equals $\beta$, or tends to infinity according as $0<\gamma<1$,
$\gamma=1$, or $\gamma>1$.  These alternatives give, respectively, a
gamma total with an independent $\PD(\theta)$ partition, the discrete law
$\QQ_\beta$, and concentration of $R_x$ at one.  Theorem~\ref{thm:joint}
adds the cutoff location.  Thus the two transitions in the title refer to
different scales: $\gamma=0$ changes the macroscopic partition, whereas
$\gamma=1$ changes the residual scale after the giant prime is present.

\emph{The normalizing constant.}  As a companion to the cofactor law, the
same argument gives $Z_x$ in all
four regimes.  At $\theta=\beta=\gamma=1$ it recovers
$\sum_{n\le x}P^+(n)\sim(\pi^2/12)x^{2}/\log x$ of Alladi and Erd\H{o}s
\cite{AlladiErdos1977}.  Jakimczuk \cite{Jakimczuk2013} gave a shorter
direct proof of the corresponding asymptotic for
$\sum_{n\le x}P^+(n)^k$ when $k$ is a fixed positive integer.
Formula~\eqref{eq:Z-unified} allows every real $\beta>0$ and includes the
generalized divisor weight $d_\theta$.

\subsection{The prime number theorem mechanism and proof strategy}
\label{sec:mechanism}

The prime number theorem does more here than estimate the normalizing
constant: it determines the reference scale against which the external
field acts.  To see this, temporarily replace $V_\gamma$ by a smooth
increasing function $V:[0,\infty)\to\R$.  On the coordinate $s=\log p$, the pushforward
of prime counting measure has local asymptotic density $e^s\,ds/s$.
For fields with regular growth, this suggests the Laplace approximation
\[
 \sum_{p\le x}\exp\{V(\log p)\}
 \approx \frac{x\exp\{V(\log x)\}}
 {\log x\{1+V'(\log x)\}}.
\]
For $V=V_\gamma$, Lemma~\ref{lem:weighted-pnt} proves this asymptotic,
and Lemma~\ref{lem:weighted-prime-ratios} gives the uniform ratios
needed below.  Thus the prime number theorem supplies the arithmetic
density of states on which the field acts.

Write $n=pm$, where $p=P^+(n)$, and put $L=\log x$ and $r=\log m$.
The decomposition by the largest prime reduces the weight associated with
$m$ to $d_\theta(m)$ times a weighted prime sum with upper bound $x/m$.
After removing factors that do not depend on $m$, the preceding
approximation gives
\[
 d_\theta(m)m^{-1}\exp\{V(L-r)-V(L)\},
\]
up to slowly varying factors.  The factor $m^{-1}$ is forced by the prime
number theorem.  It is the critical Euler product weight, and the external
field changes it through the loss $V(L)-V(L-r)$.

Two scales of this loss govern the phase diagram.  Return to
$V=V_\gamma$.  On the macroscopic scale $r=\delta L$, for fixed
$\delta\in(0,1)$,
\[
 V(L)-V((1-\delta)L)
 =\beta\{1-(1-\delta)^\gamma\}L^\gamma.
\]
For every fixed $\gamma>0$, the resulting exponential penalty dominates
the polynomial growth in $L$ of the accumulated critical weight.  This
predicts that $r/L$ vanishes and hence that one prime carries
asymptotically all logarithmic mass.  At $\gamma=0$ the penalty is absent,
and the Poisson--Dirichlet partition remains.

Once $r=o(L)$, the local behavior of the same field becomes relevant:
\[
 V(L-r)-V(L)=-a_xr+O\left(\frac{a_xr^2}{L}\right),
 \qquad a_x=V'(L)=\beta\gamma L^{\gamma-1}.
\]
If this remainder is negligible on the scale of the
cofactor, its effective weight is $d_\theta(m)m^{-1-a_x}$, precisely the
Euler product law $\QQ_{a_x}$ in \eqref{eq:Qa-intro}.  The limits
$a_x\to0$, $a_x=\beta$, and $a_x\to\infty$ distinguish the three positive
regimes.  The first transition is therefore detected by a macroscopic
increment of the field, whereas the second is detected by its local
derivative.  Both reflect the interaction of the same field with the
prime density.

To prove these limits, we combine an exact largest prime decomposition
with uniform weighted prime estimates, an $L^1(\QQ_{a_x})$ comparison,
and analyses of the moving Euler product law and the cutoff location.

\subsection{Related work}\label{sec:related}

The closest arithmetic precedents come from multiplicative reference
weights, greatest prime factor sums, and zeta laws.  Elboim and Gorodetsky
\cite{ElboimGorodetsky2024} establish the Poisson--Dirichlet limit for a
broad class of multiplicative weights that includes $d_\theta$.  Alladi and
Erd\H{o}s \cite{AlladiErdos1977,AlladiErdos1979} initiated the study of
large greatest prime factor sums; the exponent one asymptotic used here is
in \cite{AlladiErdos1977}, and Jakimczuk \cite{Jakimczuk2013} gives a
shorter direct proof for fixed positive integer powers.  Related sums with additive
functions were studied by De Koninck, K\'atai and Mercier
\cite{DeKoninckKataiMercier1989}.  At $\theta=1$, $\QQ_a$ is the usual
zeta law on the positive integers; if $M_a\sim\QQ_a$, then $-\log M_a$
has the Riemann zeta distribution studied by Lin and Hu
\cite{LinHu2001}.  Fixed rank asymptotics in the zeta and harmonic
settings go back to Lloyd \cite{Lloyd1984}; Hirth
\cite{Hirth1997} develops a related GEM and Poisson--Dirichlet viewpoint.
The present model differs from these works because the
reward depending on $P^+(n)$ is not multiplicative; the total variation
theorem nevertheless identifies the law of the entire cofactor.

Further work on Billingsley's theorem and its probabilistic structure
includes the scale invariant Poisson representation of Arratia, Barbour and
Tavar\'e \cite{ABT1999}, Tenenbaum's quantitative estimate
\cite{Tenenbaum2000}, the direct proof of Arratia and Kochman
\cite{ArratiaKochman2014}, and the optimal order $\ell^1$ coupling estimate
of Haddad and Koukoulopoulos \cite{HaddadKoukoulopoulos2025}.  Arratia,
Kochman and Miller \cite{ArratiaKochmanMiller2014} developed a criterion
based on multi-intensities and extended the theorem to normed arithmetic
semigroups.

Recent variants include Poisson--Dirichlet limits under the regularity,
level of distribution one, and congruence uniformity assumptions of
Bharadwaj and Rodgers \cite{BharadwajRodgers2026}, and quantitative
Wasserstein approximations by attenuated Poisson--Dirichlet laws for
distinct prime divisors under harmonic sampling
\cite{BernalRamirezTorresFloresJaramillo2026}.  Dawson and Feng studied
large deviations for $\PD(\theta)$ as $\theta\to\infty$, including laws
under selection \cite{DawsonFeng2006}; see also
\cite{Feng2010}.  Here $\theta$ is fixed.

Random permutations with multiplicative cycle weights provide a natural
comparison: the permutation decomposes into cycles and the weight
factorizes over them \cite{BetzUeltschiVelenik2011,ErcolaniUeltschi2014}.
Our deformation instead rewards the largest prime factor alone.  The
Feynman cycle representation of the ideal Bose gas leads naturally to
spatial random permutations.  In this setting, Betz and Ueltschi
\cite{BetzUeltschiPD2011} proved a $\PD(\vartheta)$ limit for normalized
macroscopic cycle lengths, with $\vartheta$ determined by the cycle
weights.  K\"onig, Vogel and Zass \cite{KonigVogelZass2025} proved the
$\PD(1)$ limit for macroscopic Feynman loops in the canonical free Bose
gas.  Sun \cite{SunBEC2026} obtained more general marked
Poisson--Kingman bridge limits.  The connection with
Section~\ref{sec:euler} is most direct in the gamma process case: here the
Euler product limit is unconditioned, with a random $\Gam(\theta,1)$ total
independent of its normalized $\PD(\theta)$ jumps, whereas Sun's canonical
bridge fixes the total mass.  More general heat trace profiles need not
give gamma bridges.

\subsection*{Organization}

Section~\ref{sec:results} states the main results.
Section~\ref{sec:arithmetic} proves the weighted prime sum estimates, the
cofactor approximation, and the partition function asymptotics.
Section~\ref{sec:euler} analyzes the Euler product limit, and
Section~\ref{sec:joint} proves the joint limits and their consequences.

\section{Main results}\label{sec:results}

This section begins with the unperturbed law as a comparison, then states
the cofactor approximation and the resulting phase diagram.  The
partition function asymptotics follow afterward.  Throughout, all logarithms are
natural.  Recall that
$a_x=\beta\gamma(\log x)^{\gamma-1}$ and $b_x=1+a_x$.
We set $\R_+=[0,\infty)$ and give $\N$ the discrete topology.
Unless another parameter is indicated,
$\Longrightarrow$, $\xrightarrow{\Pp}$, and $\sim$ denote convergence in
distribution, convergence in probability, and asymptotic equivalence as
$x\to\infty$.  Constants in $O(\cdot)$ may depend on
$\theta,\beta,\gamma$, but not on $x$, and $o(1)$ denotes a quantity tending
to zero.

\subsection{The unperturbed regime}

We use the state space
\[
 \ell^1_\downarrow
 =\left\{\mathbf s=(s_i)_{i\ge1}:
 s_1\ge s_2\ge\cdots\ge0,\ \sum_is_i<\infty\right\},
\]
equipped with the $\ell^1$ norm.  The vector $\W(n)$ defined in
Section~\ref{sec:overview} lies in its unit simplex.

At $\gamma=0$ the tilt is constant away from the integer one.  The
Poisson--Dirichlet component of the following result is the generalized
Billingsley theorem of Elboim and Gorodetsky
\cite[Theorem~1.1]{ElboimGorodetsky2024}; the joint limit with $N_x/x$
and the independence are established in Section~\ref{sec:zero}.

\begin{theorem}[Unperturbed joint limit]\label{thm:zero}
Suppose $\gamma=0$.  Then
\begin{equation}\label{eq:zero-joint}
 \left(\frac{N_x}{x},\W(N_x)\right)
 \Longrightarrow(U,\bP)
\end{equation}
in $[0,1]\times\ell^1_\downarrow$, where $U\sim\Unif(0,1)$,
$\bP=(S_i)_{i\ge1}\sim\PD(\theta)$, and the two limits are independent.
\end{theorem}

\subsection{The cofactor law}

The cofactor $R_x$ is defined in Section~\ref{sec:overview}, and the Euler
product law $\QQ_a$ in \eqref{eq:Qa-intro}.  We write $\Law(X)$ for the
distribution of $X$ and, for probability measures on a countable space,
use
\[
 d_{\TV}(\mu,\nu)=\frac12\sum_m|\mu(m)-\nu(m)|.
\]

For every positive exponent, the cofactor has a single asymptotic law,
which covers all three regimes at once.

\begin{theorem}[Cofactor approximation for $\gamma>0$]
\label{thm:cofactor}
Suppose $\gamma>0$.  Then, as $x\to\infty$,
\begin{equation}\label{eq:TV-main}
 d_{\TV}\bigl(\Law(R_x),\QQ_{a_x}\bigr)
 \longrightarrow0.
\end{equation}
\end{theorem}

The approximation is in absolute total variation; in the supercritical
regime it does not assert relative asymptotics for rare composite events.

Theorem~\ref{thm:cofactor} determines the marginal cofactor behavior by
reducing it to $\QQ_{a_x}$.  Section~\ref{sec:euler} identifies the limit
of this Euler product law when $a_x\to0$; the cases $a_x=\beta$ and
$a_x\to\infty$ follow directly from its definition.  The next theorem
adds the cutoff location and its joint behavior with the cofactor.  Its
proof in Section~\ref{sec:joint} uses the weighted prime sum ratios from
Lemma~\ref{lem:weighted-prime-ratios}.

\subsection{The phase diagram}

Together with Theorem~\ref{thm:zero}, the next theorem gives the
regime specific joint limits.  For $\gamma>0$, the regimes differ through
the limiting behavior of the boundary slope $a_x$.  The theorem adds the
cutoff location $N_x/x$ to the cofactor limits.  Each state space below
has its product topology.

\begin{theorem}[The phase diagram for $\gamma>0$]\label{thm:joint}
Suppose $\gamma>0$.

\begin{enumerate}[label=\textup{(\roman*)},leftmargin=9mm]
\item If $0<\gamma<1$, then
\begin{equation}\label{eq:joint-limit}
 \left(
 \frac{N_x}{x},
 a_x\log R_x,
 \W(R_x)
 \right)
 \Longrightarrow
 (U,G,\bP),
\end{equation}
in $[0,1]\times\R_+\times\ell^1_\downarrow$, where
$U\sim\Unif(0,1)$, $G\sim\Gam(\theta,1)$, $\bP\sim\PD(\theta)$,
and the three limits are mutually independent.

\item If $\gamma=1$, then
\begin{equation}\label{eq:critical-joint}
 \left(\frac{N_x}{x},R_x\right)
 \Longrightarrow(B_\beta,M_\beta),
\end{equation}
in $[0,1]\times\N$, where $M_\beta\sim\QQ_\beta$,
\[
 \Pp(B_\beta\le u)=u^{1+\beta},
 \qquad 0\le u\le1,
\]
and $B_\beta$ and $M_\beta$ are independent.

\item If $\gamma>1$, then
\begin{equation}\label{eq:super-boundary}
 \left(b_x\log\frac{x}{N_x},R_x\right)
 \Longrightarrow(E,1),
\end{equation}
in $\R_+\times\N$, where $E\sim\Exp(1)$ has rate one.  Consequently
$N_x/x\to1$ in
probability.
\end{enumerate}
\end{theorem}

The three positive regimes share the following macroscopic conclusion.

\begin{corollary}[Giant prime]\label{cor:giant-prime}
For every $\gamma>0$,
\begin{equation}\label{eq:giant}
 \frac{\log P^+(N_x)}{\log N_x}
 \xrightarrow{\Pp}1,
\end{equation}
where the ratio is assigned the value one on $\{N_x=1\}$, and
\begin{equation}\label{eq:full-degenerate}
 \W(N_x)
 \xrightarrow{\Pp}(1,0,0,\ldots)
 \quad\text{in }\ell^1_\downarrow.
\end{equation}
\end{corollary}

The following consequence displays both thresholds on the scale of the
missing logarithmic mass.  Let
\[
 \mathcal D_x=
 \begin{cases}
  1-\dfrac{\log P^+(N_x)}{\log N_x},&N_x\ge2,\\[2mm]
  0,&N_x=1.
 \end{cases}
\]

\begin{corollary}[Scale of the missing logarithmic mass]\label{cor:defect}
Let $(S_i)_{i\ge1}\sim\PD(\theta)$, and let $G$ and $M_\beta$ be as in
Theorem~\textup{\ref{thm:joint}}.  Then
\[
 \begin{array}{c|c}
 \gamma=0
 &\mathcal D_x\Longrightarrow1-S_1,\\[1mm]
 0<\gamma<1
 &(\log x)^\gamma \mathcal D_x\Longrightarrow G/(\beta\gamma),\\[1mm]
 \gamma=1
 &(\log x)\mathcal D_x\Longrightarrow\log M_\beta,\\[1mm]
 \gamma>1
 &\Pp_x(\mathcal D_x=0)\longrightarrow1.
 \end{array}
\]
On the ordinary multiplicative scale,
\begin{equation}\label{eq:ordinary-ratio-regimes}
 \frac{P^+(N_x)}{N_x}
 \Longrightarrow
 \begin{cases}
  0,&0\le\gamma<1,\\
  M_\beta^{-1},&\gamma=1,\\
  1,&\gamma>1.
 \end{cases}
\end{equation}
\end{corollary}

\subsection{The partition function}

\begin{theorem}[Partition function asymptotics]\label{thm:partition}
If $\gamma=0$, then
\begin{equation}\label{eq:Z-zero}
 Z_x
 \sim
 \frac{e^\beta}{\Gamma(\theta)}
 x(\log x)^{\theta-1}.
\end{equation}
If $\gamma>0$, then
\begin{equation}\label{eq:Z-unified}
 Z_x
 \sim
 \frac{\theta x\exp\{\beta(\log x)^\gamma\}}
 {(\log x)b_x}\,\zeta(1+a_x)^\theta.
\end{equation}
Equivalently, the three regimes with positive exponent are
\begin{equation}\label{eq:Z-regimes}
 Z_x\sim
 \begin{cases}
 \theta(\beta\gamma)^{-\theta}
 x e^{\beta(\log x)^\gamma}
 (\log x)^{\theta(1-\gamma)-1},&0<\gamma<1,\\[1mm]
 \dfrac{\theta}{1+\beta}\zeta(1+\beta)^\theta
 \dfrac{x^{1+\beta}}{\log x},&\gamma=1,\\[3mm]
 \dfrac{\theta}{\beta\gamma}
 x e^{\beta(\log x)^\gamma}(\log x)^{-\gamma},&\gamma>1.
 \end{cases}
\end{equation}
\end{theorem}

At $\theta=\beta=\gamma=1$, Theorem~\ref{thm:partition} recovers the
classical asymptotic of Alladi and Erd\H{o}s \cite{AlladiErdos1977},
\[
 \sum_{n\le x}P^+(n)
 \sim \frac{\pi^2}{12}\frac{x^2}{\log x}.
\]
The three lines of \eqref{eq:Z-regimes} follow from \eqref{eq:Z-unified}
and the behavior of $\zeta$ near $1$: $\zeta(1+a)\sim a^{-1}$ as
$a\downarrow0$, and $\zeta(1+a)\to1$ as $a\to\infty$.

The probability that $N_x$ is prime distinguishes the four regimes
quantitatively.

\begin{corollary}[Probability that $N_x$ is prime]\label{cor:prime-freezing}
For fixed $\theta,\beta>0$,
\[
 \Pp_x(N_x\text{ is prime})\longrightarrow
 \begin{cases}
  0,&0\le\gamma<1,\\
  \zeta(1+\beta)^{-\theta},&\gamma=1,\\
 1,&\gamma>1.
 \end{cases}
\]
More precisely,
\[
 \Pp_x(N_x\text{ is prime})
 \sim
 \begin{cases}
  \Gamma(\theta+1)(\log x)^{-\theta},&\gamma=0,\\[1mm]
  a_x^\theta=(\beta\gamma)^\theta(\log x)^{-\theta(1-\gamma)},&0<\gamma<1.
 \end{cases}
\]
\end{corollary}

\section{The cofactor law and the partition function}\label{sec:arithmetic}

This section proves Theorems~\ref{thm:cofactor}
and~\ref{thm:partition}.  The reward is not multiplicative, but decomposing
at a largest prime greater than $\sqrt x$ separates that prime from its
cofactor.  A weighted prime number theorem estimate and an $L^1(\QQ_a)$
comparison then identify both the normalized cofactor weights and the total
mass.  The final subsection combines these ingredients and treats
$\gamma=0$ separately.  Throughout the analysis
for $\gamma>0$, we write $L=\log x$ and abbreviate
\[
 a=a_x=\beta\gamma L^{\gamma-1},
 \qquad
 1+a=b_x,
 \qquad
 C_x=\frac{x e^{\beta L^\gamma}}{L(1+a)}.
\]

\subsection{Partition function decomposition}

Let $Z_x^{>}$ and $Z_x^{\le}$ be the contributions to
\eqref{eq:partition} from the integers with $P^+(n)>\sqrt x$ and with
$P^+(n)\le\sqrt x$, respectively.  The first part admits an exact
decomposition and supplies the main term in \eqref{eq:Z-unified}; the
complementary part is negligible on the same scale.  The cutoff at
$\sqrt x$ also ensures that the largest prime occurs exactly once.  The
weighted prime sum that appears in the decomposition is
\[
 A(y)=\sum_{p\le y}\exp\{\beta(\log p)^\gamma\},
 \qquad y>0,
\]
so that $A(y)=0$ for $y<2$.  We shall also use the following standard
summatory estimate.

\begin{lemma}[Summatory estimate for $d_\theta$]\label{lem:dtheta-sum}
As $y\to\infty$,
\begin{equation}\label{eq:selberg-delange}
 D_\theta(y):=\sum_{n\le y}d_\theta(n)
 \sim \frac{y(\log y)^{\theta-1}}{\Gamma(\theta)}.
\end{equation}
There is a constant $C$ such that
\[
 D_\theta(y)\le Cy(\log y)^{\theta-1},\qquad y\ge2,
\]
and, for every fixed $u\in(0,1]$,
\[
 \frac{D_\theta(uy)}{D_\theta(y)}\longrightarrow u.
\]
\end{lemma}

\begin{proof}
The asymptotic \eqref{eq:selberg-delange} is the case
$\alpha=d_\theta$ and $d=0$ of Elboim and Gorodetsky
\cite[Corollary~3.2]{ElboimGorodetsky2024}; it also follows directly from
the Selberg--Delange method, as in Tenenbaum
\cite{Tenenbaum2015}.  Enlarging the constant in the
asymptotic estimate gives the upper bound.  Finally,
\[
 \frac{D_\theta(uy)}{D_\theta(y)}
 \sim u\left(\frac{\log(uy)}{\log y}\right)^{\theta-1}
 \longrightarrow u.
\]
\end{proof}

\begin{lemma}[Cofactor decomposition and the small sector]
\label{lem:exact-decomp}
Suppose $\gamma>0$.  Then
\begin{equation}\label{eq:exact-decomp}
 Z_x^{>}
 =
 \theta\sum_{m<\sqrt x}d_\theta(m)
 \bigl(A(x/m)-A(\sqrt x)\bigr).
\end{equation}
Moreover,
\begin{equation}\label{eq:small-sector}
 Z_x^{\le}
 =
 o\left(
 \frac{x\exp\{\beta(\log x)^\gamma\}}
 {(\log x)b_x}\,\zeta(1+a_x)^\theta
 \right).
\end{equation}
\end{lemma}

\begin{proof}
We first compute $Z_x^>$.  For every integer $n$ counted by $Z_x^>$, set
\[
 p=P^+(n),
 \qquad
 m=n/p.
\]
Since $p>\sqrt x$, we have $m=n/p\le x/p<\sqrt x<p$, and hence
$(m,p)=1$.  Therefore
\[
 d_\theta(pm)=d_\theta(p)d_\theta(m)
 =\theta d_\theta(m).
\]
Conversely, if $m<\sqrt x$ and $p$ is a prime with $\sqrt x<p\le x/m$,
then $n=mp\le x$, and every prime $q\mid m$ satisfies
$q\le m<\sqrt x<p$, so $P^+(n)=p$.  This gives a bijection between the
integers counted by $Z_x^>$ and these pairs $(m,p)$.  Using the definition
of $A$,
\begin{align*}
 Z_x^>
 &=\sum_{m<\sqrt x}\
   \sum_{\sqrt x<p\le x/m}
   d_\theta(mp)e^{\beta(\log p)^\gamma}\\
 &=\theta\sum_{m<\sqrt x}d_\theta(m)
   \sum_{\sqrt x<p\le x/m}e^{\beta(\log p)^\gamma}
 =\theta\sum_{m<\sqrt x}d_\theta(m)
   \bigl(A(x/m)-A(\sqrt x)\bigr),
\end{align*}
which proves \eqref{eq:exact-decomp}.

We next bound $Z_x^\le$.  If $P^+(n)\le\sqrt x$, then
$\log P^+(n)\le L/2$ and
$e^{\beta(\log P^+(n))^\gamma}\le e^{\beta2^{-\gamma}L^\gamma}$.  The
upper bound in Lemma~\ref{lem:dtheta-sum} now gives
\[
 Z_x^{\le}
 \le
 e^{\beta2^{-\gamma}L^\gamma}D_\theta(x)
 \le
 CxL^{\theta-1}e^{\beta2^{-\gamma}L^\gamma}.
\]
Put $c_\gamma=\beta(1-2^{-\gamma})>0$.  Then
\[
 \frac{Z_x^\le}{C_x\zeta(1+a)^\theta}
 \le
 C L^\theta\frac{1+a}{\zeta(1+a)^\theta}
 e^{-c_\gamma L^\gamma}.
\]
Since $\zeta(1+a)\ge1$ and
$1+a=O(L^{\max\{\gamma-1,0\}})$,
\[
 \frac{Z_x^\le}{C_x\zeta(1+a)^\theta}
 \le C L^{\theta+\max\{\gamma-1,0\}}e^{-c_\gamma L^\gamma}
 \longrightarrow0.
\]
This is \eqref{eq:small-sector}.
\end{proof}

\subsection{Estimates for the weighted prime sum
\texorpdfstring{$A$}{A}}

By \eqref{eq:exact-decomp}, it remains to estimate $A(x/m)$.  We first
obtain an asymptotic formula for $A$ and then derive two uniform ratio
estimates.

For $s>0$, put
\[
 F(s)=s+\beta s^\gamma,
 \qquad
 b(s)=F'(s)=1+\beta\gamma s^{\gamma-1},
\]
and, for $y>1$, write $b_y=b(\log y)$.
In particular, $b_x=b(\log x)$.

\begin{lemma}[Asymptotic estimate for the weighted prime sum]
\label{lem:weighted-pnt}
Let $\gamma>0$.  Then, as $y\to\infty$,
\begin{equation}\label{eq:weighted-pnt}
 A(y)
 =\frac{y\exp\{\beta(\log y)^\gamma\}}
 {(\log y)b_y}
 \left(1+O\left(\frac1{(\log y)b_y}\right)\right).
\end{equation}
In particular, there is a constant $C$ such that
\begin{equation}\label{eq:weighted-pnt-upper}
 A(y)\le C\frac{y\exp\{\beta(\log y)^\gamma\}}
 {(\log y)b_y},
 \qquad y\ge e.
\end{equation}
\end{lemma}

\begin{proof}
\smallskip
\noindent\emph{Step 1. The prime number theorem and Stieltjes
integration.}
Set
\[
 w(t)=e^{\beta(\log t)^\gamma},
 \qquad
 w'(t)=\frac{\beta\gamma(\log t)^{\gamma-1}}{t}w(t),
 \qquad t\ge2.
\]
Let
\[
 \pi(t)=\#\{p:p\le t\},
 \qquad
 \Li(t)=\int_2^t\frac{du}{\log u},
 \qquad
 \mathcal E(t)=\pi(t)-\Li(t),
 \qquad t\ge2,
\]
and extend all three functions by zero on $(0,2)$.  The prime number
theorem with the error term coming from the classical zero free region
gives
\[
 \mathcal E(t)=O\bigl(te^{-c\sqrt{\log t}}\bigr)
\]
for some $c>0$; see Montgomery and Vaughan
\cite{MontgomeryVaughan2006}.  Stieltjes integration by parts
and $d\Li(t)=dt/\log t$ give
\begin{equation}\label{eq:A-stieltjes}
 A(y)=\int_{2^-}^{y}w(t)\,d\pi(t)
 =\int_2^{y}\frac{w(t)}{\log t}\,dt
 +w(y)\mathcal E(y)-\int_2^{y}\mathcal E(t)w'(t)\,dt.
\end{equation}
Thus the prime number theorem separates $A(y)$ into a continuous main
integral and two error terms, which will be estimated in Step~3.
\par\smallskip
\noindent\emph{Step 2. The logarithmic change of variables and
the Laplace main term.}
The substitution $t=e^s$ gives
\[
 \int_2^y\frac{w(t)}{\log t}\,dt
 =\int_{\log2}^{\log y}\frac{e^{F(s)}}s\,ds.
\]
Since $F'(\log y)=b(\log y)$, the exponent changes by an amount of order
one when $\log y-s$ is of order $1/b(\log y)$.  The integral is therefore
concentrated in a layer of this width, while $1/s$ is nearly constant
there.  Laplace's method consequently suggests
\[
 \int_{\log2}^{\log y}\frac{e^{F(s)}}s\,ds
 \sim\frac{e^{F(\log y)}}{(\log y)b(\log y)}.
\]
We prove this prediction directly.  Since
$d(e^{F(s)})=b(s)e^{F(s)}\,ds$, integration by parts gives
\begin{align*}
 \int_{\log2}^{\log y}\frac{e^{F(s)}}s\,ds
 &=\frac{e^{F(\log y)}}{(\log y)b(\log y)}
 -\frac{e^{F(\log2)}}{(\log2)b(\log2)}\\
 &\quad+
 \int_{\log2}^{\log y}
 \frac{1+\beta\gamma^2s^{\gamma-1}}{s^2b(s)^2}e^{F(s)}\,ds.
\end{align*}
The first term is the predicted Laplace term.  Step~3 bounds the
constant term and the remaining integral.

\smallskip
\noindent\emph{Step 3. Remainder bounds.}
We first show that the integral remainder in Step~2 is smaller than the
Laplace term.  We then control the two prime number theorem errors in
\eqref{eq:A-stieltjes}.  The coefficient of $e^{F(s)}$ in the integral
remainder satisfies
\[
 -\left(\frac1{sb(s)}\right)'
 =\frac{1+\beta\gamma^2s^{\gamma-1}}{s^2b(s)^2}>0,
\]
while
\[
 1+\beta\gamma^2s^{\gamma-1}
 \le\max\{1,\gamma\}b(s).
\]
For all sufficiently large $y$, split the last integral in the
integration by parts formula at $(\log y)/2$.  On the upper part,
$sb(s)=s+\beta\gamma s^\gamma$ is comparable with
$(\log y)b(\log y)$, and hence
\[
 0\le
 \frac{1+\beta\gamma^2s^{\gamma-1}}{s^2b(s)^2}
 \le\frac{C}{(\log y)b(\log y)}\frac1s.
\]
On the lower part, the identity for the derivative of $1/(sb(s))$ and
the monotonicity of $F$ give
\[
 \int_{\log2}^{(\log y)/2}
 \frac{1+\beta\gamma^2s^{\gamma-1}}{s^2b(s)^2}e^{F(s)}\,ds
 \le\frac{e^{F((\log y)/2)}}{(\log2)b(\log2)}.
\]
Since
$F(\log y)-F((\log y)/2)\ge(\log y)/2$ and
$(\log y)b(\log y)$ grows polynomially in $\log y$, the preceding
integration by parts identity gives
\[
 \int_{\log2}^{\log y}\frac{e^{F(s)}}s\,ds
 \le \frac{e^{F(\log y)}}{(\log y)b(\log y)}
 +o\left(\frac{e^{F(\log y)}}{(\log y)b(\log y)}\right)
 +\frac{C}{(\log y)b(\log y)}
 \int_{\log2}^{\log y}\frac{e^{F(s)}}s\,ds.
\]
The coefficient of the last integral is smaller than $1/2$ for large
$y$.  Moving that term to the left shows that
\[
 \int_{\log2}^{\log y}\frac{e^{F(s)}}s\,ds
 =O\left(\frac{e^{F(\log y)}}{(\log y)b(\log y)}\right).
\]
The bounded lower limit term and the part of the remainder below
$(\log y)/2$ are exponentially smaller than the Laplace term.  On the
upper part, substituting the preceding bound into the remainder estimate
gives
\[
 \int_{\log2}^{\log y}\frac{e^{F(s)}}s\,ds
 =\frac{e^{F(\log y)}}{(\log y)b(\log y)}
 \left(1+O\left(\frac1{(\log y)b(\log y)}\right)\right).
\]
We now return to the two error terms in \eqref{eq:A-stieltjes}.  For all
sufficiently large $s$,
\[
 \left(F(s)-c\sqrt s\right)'\ge\frac12b(s),
 \qquad
 s^{\gamma-1}\le C_{\beta,\gamma}b(s).
\]
Choose $s_0$ so that these bounds hold.  After substituting $t=e^s$ in
the last error term in \eqref{eq:A-stieltjes}, we obtain, for some
$c_1>0$,
\begin{align*}
 |w(y)\mathcal E(y)|
 &\le C e^{F(\log y)-c\sqrt{\log y}},\\
 \left|\int_2^{y}\mathcal E(t)w'(t)\,dt\right|
 &\le C+C\int_{s_0}^{\log y}
 e^{F(s)-c\sqrt s}\left(F(s)-c\sqrt s\right)'\,ds
 \le C e^{F(\log y)-c_1\sqrt{\log y}}.
\end{align*}
Both errors are smaller than the leading term times the relative error
displayed above, since
\[
 \frac{e^{F(\log y)-c_1\sqrt{\log y}}}
 {e^{F(\log y)}/\{(\log y)b(\log y)\}}
 =(\log y)b(\log y)e^{-c_1\sqrt{\log y}}
 =o\left(\frac1{(\log y)b(\log y)}\right).
\]
Consequently,
\[
 A(y)=
 \frac{y\exp\{\beta(\log y)^\gamma\}}
 {(\log y)\{1+\beta\gamma(\log y)^{\gamma-1}\}}
 \left(1+O\left(
 \frac1{(\log y)\{1+\beta\gamma(\log y)^{\gamma-1}\}}
 \right)\right),
\]
which proves \eqref{eq:weighted-pnt}.  The upper bound
\eqref{eq:weighted-pnt-upper} follows for large $y$ from this estimate;
enlarging the constant covers the remaining compact interval.

The relative error above is uniformly small for $y\ge Y$ as
$Y\to\infty$.  In the applications below, $y=x/m$ with $m<\sqrt x$,
so $y>\sqrt x$ and the estimate is uniform in $m$.
\end{proof}

\begin{lemma}[Uniform ratio estimates for $A$]
\label{lem:weighted-prime-ratios}
Let $\gamma>0$.  For every fixed $0<T<\infty$,
\begin{equation}\label{eq:prime-boundary-ratio}
 \sup_{0\le t\le T}
 \left|
 \frac{A\bigl(ye^{-t/b_y}\bigr)}{A(y)}-e^{-t}
 \right|\longrightarrow0,
 \qquad y\to\infty.
\end{equation}
If $0<\gamma\le1$, then for every compact $K\subset(0,1]$,
\begin{equation}\label{eq:prime-ratio-u}
 \frac{A(uy)}{A(y)}\longrightarrow
 \begin{cases}
  u,&0<\gamma<1,\\
  u^{1+\beta},&\gamma=1,
 \end{cases}
 \qquad\text{uniformly for }u\in K,
\end{equation}
\end{lemma}

\begin{proof}
For $0\le t\le T$, set $\rho=t/b(\log y)$, so that $0\le\rho\le T$.  Since
$b'(v)=\beta\gamma(\gamma-1)v^{\gamma-2}$, we have
$\sup_{\log y-\rho\le v\le\log y}|b'(v)|
\le C(\log y)^{\gamma-2}$ once $\log y\ge2T$.  Hence
\[
 |F(\log y-\rho)-F(\log y)+t|
 =\left|\int_{\log y-\rho}^{\log y}
 \{b(\log y)-b(v)\}\,dv\right|
 \le C_T\frac{(\log y)^{\gamma-2}}{b(\log y)^2}=o(1),
\]
uniformly in $t$; the last step uses
$b(\log y)\ge
\max\{1,\beta\gamma(\log y)^{\gamma-1}\}$, so that
\[
 \frac{(\log y)^{\gamma-2}}{b(\log y)^2}
 \le\min\bigl\{(\log y)^{\gamma-2},
 (\beta\gamma)^{-2}(\log y)^{-\gamma}\bigr\}
 \longrightarrow0.
\]
Also, $\rho=O_T(1)$ and the preceding bound on $b'$ show that, uniformly
in $t$,
\[
 \frac{(\log y)b(\log y)}
 {(\log y-\rho)b(\log y-\rho)}=1+o(1).
\]
Because $ye^{-\rho}\ge ye^{-T}\to\infty$, applying
\eqref{eq:weighted-pnt} at $ye^{-\rho}$ and at $y$ gives, uniformly in
$t$,
\[
 \frac{A(ye^{-\rho})}{A(y)}
 =e^{F(\log y-\rho)-F(\log y)}
   \frac{(\log y)b(\log y)}
   {(\log y-\rho)b(\log y-\rho)}\bigl(1+o(1)\bigr)
 =e^{-t}+o(1),
\]
uniformly for $0\le t\le T$.  This proves
\eqref{eq:prime-boundary-ratio}.

Finally, let $0<\gamma\le1$ and let $K\subset(0,1]$ be compact.
Uniformly for $u\in K$, the mean value theorem gives
\begin{align*}
 F(\log y+\log u)-F(\log y)
 &=\begin{cases}
    \log u+o(1),&0<\gamma<1,\\
    (1+\beta)\log u,&\gamma=1,
   \end{cases}\\
 \frac{(\log y)b(\log y)}
 {(\log y+\log u)b(\log y+\log u)}
 &=1+o(1),
 \qquad y\to\infty.
\end{align*}
The remainder in the first relation is
$O_K((\log y)^{\gamma-1})$ when $0<\gamma<1$.  The asymptotic
\eqref{eq:weighted-pnt} applies uniformly to $uy$, since
$uy\ge(\min K)y\to\infty$, and yields
\[
 \frac{A(uy)}{A(y)}
 =\begin{cases}
  u+o(1),&0<\gamma<1,\\
  u^{1+\beta}+o(1),&\gamma=1,
 \end{cases}
\]
uniformly for $u\in K$.  This proves \eqref{eq:prime-ratio-u}.
\end{proof}

\subsection{Comparison of the cofactor law with
\texorpdfstring{$\QQ_a$}{Q(a)}}

Let
\[
 \mu_x^{>}=\Law\bigl(R_x\mid P^+(N_x)>\sqrt x\bigr).
\]
The aim is to compare $\mu_x^>$ with $\QQ_a$ and simultaneously
determine the contribution $Z_x^>$.  We first give the tail estimate
under $\QQ_a$ used in the comparison.

\begin{lemma}[Exponential moment bound under $\QQ_a$]
\label{lem:Qa-exp-moment}
Let $M_a\sim\QQ_a$.  For every fixed $\eta\in(0,1)$, there is $a_0>0$
such that
\begin{equation}\label{eq:Qa-exp-moment}
 \sup_{0<a<a_0}
 \E e^{\eta a\log M_a}<\infty.
\end{equation}
Consequently $a\log M_a$ is tight as $a\downarrow0$.
\end{lemma}

\begin{proof}
The Dirichlet series \eqref{eq:dtheta-euler}, evaluated at
$1+(1-\eta)a>1$, gives
\[
 \E e^{\eta a\log M_a}
 =
 \left(
 \frac{\zeta(1+(1-\eta)a)}{\zeta(1+a)}
 \right)^\theta.
\]
Since $a\zeta(1+a)\to1$ as $a\downarrow0$, the ratio tends to
$(1-\eta)^{-\theta}$ and is therefore bounded for small $a$.
Markov's inequality then gives the asserted tightness.
\end{proof}

\begin{proposition}[Cofactor comparison with $\QQ_a$]
\label{prop:cofactor-comparison}
Suppose $0<\gamma\le1$ and set $a=a_x$.  Then
\[
 Z_x^>\sim\theta C_x\zeta(1+a)^\theta,
 \qquad
 d_{\TV}(\mu_x^>,\QQ_a)\longrightarrow0.
\]
\end{proposition}

\begin{proof}
By Lemma~\ref{lem:exact-decomp}, the Gibbs weight assigned to the
cofactor value $m<\sqrt x$, before normalization, is
\[
 \theta d_\theta(m)\bigl(A(x/m)-A(\sqrt x)\bigr).
\]
For $m\ge1$, define
\[
 g_x(m)
 =\one_{\{m<\sqrt x\}}
 \frac{A(x/m)-A(\sqrt x)}{C_xm^{-1-a}}.
\]
Let $M_a\sim\QQ_a$ and put
$\bar g_x=\E_{\QQ_a}g_x(M_a)$.  Since
$\QQ_a(m)=d_\theta(m)m^{-1-a}/\zeta(1+a)^\theta$, we have the exact
identities
\begin{equation}\label{eq:Z-density}
 Z_x^>=\theta C_x\zeta(1+a)^\theta\bar g_x,
 \qquad
 \frac{d\mu_x^>}{d\QQ_a}(m)=\frac{g_x(m)}{\bar g_x}.
\end{equation}
It remains to prove the density estimate
\begin{equation}\label{eq:gx-L1}
 \E_{\QQ_a}|g_x(M_a)-1|\longrightarrow0.
\end{equation}
First suppose $\gamma=1$, so that $a=\beta$ is fixed.  For every fixed
$m$, Lemma~\ref{lem:weighted-pnt} gives
\[
 \frac{A(x/m)}{C_xm^{-1-\beta}}
 =\frac{L}{L-\log m}\bigl(1+o(1)\bigr)\longrightarrow1,
 \qquad
 \frac{A(\sqrt x)}{C_xm^{-1-\beta}}\longrightarrow0.
\]
Thus $g_x(m)\to1$.  Uniformly for $m<\sqrt x$, the upper bound
\eqref{eq:weighted-pnt-upper} gives
\[
 0\le g_x(m)\le C\frac{L}{L-\log m}\le2C.
\]
Also $g_x(m)=0$ for $m\ge\sqrt x$.  Since $\QQ_\beta$ is a fixed
probability measure, dominated convergence proves \eqref{eq:gx-L1} in
the critical case.

Now suppose $0<\gamma<1$.  Choose
\[
 K_L=4\log L,
 \qquad
 T_L=K_L/a.
\]
Then $T_L=o(L)$, and in particular $T_L<L/2$ for all sufficiently large
$x$.  For $0\le r<L$, put
\[
 \Delta_L(r)
 =\beta\bigl((L-r)^\gamma-L^\gamma\bigr)+ar.
\]
Concavity gives
\[
 (L-r)^\gamma-L^\gamma
 \le-\gamma L^{\gamma-1}r,
 \qquad
 \Delta_L(r)\le0.
\]
Moreover, $K_L^2L^{-\gamma}\to0$.  Taylor's theorem therefore gives
\[
 \sup_{0\le r\le T_L}|\Delta_L(r)|
 \le C L^{\gamma-2}T_L^2
 \le C K_L^2L^{-\gamma}
 \longrightarrow0.
\]
Uniformly for $r=\log m\le T_L$, Lemma~\ref{lem:weighted-pnt} gives
\[
 \frac{A(x/m)}{C_xm^{-1-a}}
 =
 \frac{L}{L-r}
 \frac{1+a}{1+a(1-r/L)^{\gamma-1}}
 \exp\{\Delta_L(r)\}(1+o(1))
 \longrightarrow1.
\]
The convergence is uniform because the preceding estimate controls the
exponential factor and $T_L=o(L)$ controls both fractions.  Also,
\[
 \log\frac{A(\sqrt x)}{A(x/m)}
 =-\frac{L}{2}+r
 +\beta\bigl((L/2)^\gamma-(L-r)^\gamma\bigr)+O(\log L)
 =-\frac{L}{2}+o(L).
\]
Here the $O(\log L)$ term comes from the logarithms of the prefactors in
\eqref{eq:weighted-pnt}.  Consequently,
\begin{equation}\label{eq:A-sqrt-negligible}
 \sup_{\log m\le T_L}\frac{A(\sqrt x)}{A(x/m)}\longrightarrow0.
\end{equation}
Hence
\begin{equation}\label{eq:gx-local-uniform}
 \sup_{\log m\le T_L}|g_x(m)-1|\longrightarrow0.
\end{equation}

For every $m<\sqrt x$, the upper bound \eqref{eq:weighted-pnt-upper} and
$\Delta_L(\log m)\le0$ give
\begin{equation}\label{eq:gx-domination}
 0\le g_x(m)
 \le C\frac{L(1+a)}
 {(L-\log m)\{1+\beta\gamma(L-\log m)^{\gamma-1}\}}
 \le2C,
\end{equation}
where the last step uses $L-\log m\ge L/2$ and
$\beta\gamma(L-\log m)^{\gamma-1}\ge a$.  By the bound
\eqref{eq:Qa-exp-moment} in Lemma~\ref{lem:Qa-exp-moment}, applied with
$\eta=1/2$, and Markov's inequality,
\[
 \QQ_a(\log M_a>T_L)
 =
 \QQ_a(a\log M_a>K_L)
 \le C e^{-K_L/2}=CL^{-2}.
\]
Combining this tail bound,
\eqref{eq:gx-local-uniform}, and \eqref{eq:gx-domination}, we obtain
\[
 \E_{\QQ_a}|g_x(M_a)-1|
 \le
 \sup_{\log m\le T_L}|g_x(m)-1|
 +(2C+1)\QQ_a(\log M_a>T_L),
\]
which proves \eqref{eq:gx-L1}.

Equation \eqref{eq:gx-L1} gives $\bar g_x\to1$.  The two identities in
\eqref{eq:Z-density} now yield
\[
 Z_x^>\sim\theta C_x\zeta(1+a)^\theta.
\]
\begin{align*}
 d_{\TV}(\mu_x^>,\QQ_a)
 &=\frac{1}{2\bar g_x}
   \E_{\QQ_a}|g_x(M_a)-\bar g_x|\\
 &\le\frac{1}{\bar g_x}
   \E_{\QQ_a}|g_x(M_a)-1|
 \longrightarrow0,
\end{align*}
because
$\E_{\QQ_a}|g_x(M_a)-\bar g_x|
\le \E_{\QQ_a}|g_x(M_a)-1|+|\bar g_x-1|
\le2\E_{\QQ_a}|g_x(M_a)-1|$.
\end{proof}

\subsection{Proof of the cofactor and partition theorems}

\begin{proof}[Proof of Theorems~\textup{\ref{thm:cofactor}} and
\textup{\ref{thm:partition}}]
We now combine the preceding estimates.  The case $\gamma=0$ is
immediate, while for positive $\gamma$ we distinguish
$0<\gamma\le1$ from $\gamma>1$.

\smallskip
\noindent\emph{The case $\gamma=0$.}
The partition function is exactly
\[
 Z_x=1+e^\beta\bigl(D_\theta(x)-1\bigr),
\]
and the Selberg--Delange estimate \eqref{eq:selberg-delange} gives
\eqref{eq:Z-zero}.

\smallskip
\noindent\emph{The cases $0<\gamma\le1$.}
Proposition~\ref{prop:cofactor-comparison} and
\eqref{eq:small-sector} give
\[
 Z_x
 =Z_x^>+Z_x^\le
 \sim
 \theta C_x\zeta(1+a)^\theta.
\]
This is \eqref{eq:Z-unified}, and the first two rows of
\eqref{eq:Z-regimes} follow as explained after
Theorem~\ref{thm:partition}.

Moreover,
\[
 \Pp_x(P^+(N_x)\le\sqrt x)
 =
 \frac{Z_x^{\le}}{Z_x}\longrightarrow0
\]
by \eqref{eq:small-sector}.  Removing this vanishing event changes the
law by at most its probability, so
\[
 d_{\TV}(\Law(R_x),\QQ_a)
 \le d_{\TV}(\mu_x^>,\QQ_a)
 +\Pp_x(P^+(N_x)\le\sqrt x),
\]
which proves \eqref{eq:TV-main} for $0<\gamma\le1$.

\smallskip
\noindent\emph{The case $\gamma>1$.}
Now
\[
 b'(s)=\beta\gamma(\gamma-1)s^{\gamma-2}>0,
\]
so $b(s)$ is increasing.  Uniformly for
$0\le r\le L/2$, Lemma~\ref{lem:weighted-pnt} gives
\[
 \frac{A(e^{L-r})}{A(e^L)}
 \le C\exp\{-b(L/2)r\}.
\]
Indeed,
\[
 \frac{Lb(L)}{(L-r)b(L-r)}\le C_{\beta,\gamma},
 \qquad 0\le r\le L/2,
\]
while
\[
 F(L)-F(L-r)=\int_{L-r}^Lb(s)\,ds\ge b(L/2)r.
\]
Since $A(\sqrt x)/A(x)\to0$, the exact decomposition
\eqref{eq:exact-decomp} and the preceding ratio bound give
\[
 0\le
 \frac{Z_x^>}{\theta A(x)}
 -\left(1-\frac{A(\sqrt x)}{A(x)}\right)
 \le C\sum_{m\ge2}d_\theta(m)m^{-b(L/2)}
 =C\bigl(\zeta(b(L/2))^\theta-1\bigr)=o(1),
\]
because $b(L/2)\to\infty$.  Thus $Z_x^>\sim\theta A(x)$.  Moreover,
\eqref{eq:weighted-pnt} gives $A(x)\sim C_x$, while
$\zeta(1+a)^\theta\to1$.  Hence \eqref{eq:small-sector} gives
$Z_x^\le=o(A(x))$, and therefore
\begin{equation}\label{eq:super-Z-prime}
 Z_x\sim\theta A(x),
 \qquad
 \Pp_x(N_x\text{ is prime})
 =\frac{\theta A(x)}{Z_x}\longrightarrow1,
\end{equation}
where we used that the total Gibbs weight of the primes $p\le x$ is
$\theta A(x)$.  This is \eqref{eq:Z-unified} and the third row of
\eqref{eq:Z-regimes}.

Equation \eqref{eq:super-Z-prime} implies
$d_{\TV}(\Law(R_x),\delta_1)\to0$, where $\delta_1$ is the point
mass at one.  On the other hand,
\[
 \QQ_a(1)=\zeta(1+a)^{-\theta}\longrightarrow1.
\]
The triangle inequality, with $\delta_1$ as the intermediate law,
proves \eqref{eq:TV-main} also for $\gamma>1$.
\end{proof}

\section{The Euler product limit}\label{sec:euler}

This section supplies the joint $\ell^1$ limit needed to transfer
Theorem~\ref{thm:cofactor} from the Euler product law $\QQ_{a_x}$ to the
cofactor $R_x$ when $0<\gamma<1$.  We first prove convergence of the
rescaled logarithmic prime factors of $M_a\sim\QQ_a$, including control of
their total mass, and then normalize the limiting gamma process.  Throughout,
$a>0$ is a free parameter, independent of $x$ and $\gamma$, and
$a\downarrow0$.  When $\theta=1$, the asymptotics of each fixed ranked
prime divisor under the zeta law go back to Lloyd \cite{Lloyd1984}; see
also the GEM and Poisson--Dirichlet perspective of Hirth
\cite{Hirth1997}.  Here we establish joint $\ell^1$ convergence of the
rescaled logarithmic prime factors and their total, independence after
normalization, and the extension to all $\theta>0$.

We use the gamma subordinator construction of $\PD(\theta)$.  Let
$(\mathcal G(t))_{t\ge0}$ be the subordinator with
$\E e^{-s\mathcal G(t)}=(1+s)^{-t}$ for $s,t\ge0$ and L\'evy measure
$e^{-u}\,du/u$ on $(0,\infty)$.  If $J_1\ge J_2\ge\cdots$ are its ranked
jumps on $[0,\theta]$, then $\sum_iJ_i=\mathcal G(\theta)$ has the
$\Gam(\theta,1)$ distribution, and
\[
 \left(\frac{J_i}{\mathcal G(\theta)}\right)_{i\ge1}
 \sim\PD(\theta).
\]
Moreover, the total $\mathcal G(\theta)$ is independent of the normalized
sequence \cite{Kingman1975}.  This construction will identify the limit
of the prime factors under $\QQ_a$.

Let $M_a\sim\QQ_a$ and write
\[
 M_a=\prod_p p^{V_{a,p}},
\]
so that $V_{a,p}=\nu_p(M_a)$ is the multiplicity of the prime $p$ in
the random integer $M_a$.  Combining \eqref{eq:Qa-intro},
\eqref{eq:dtheta-euler}, and \eqref{eq:dtheta-local}, for every finite
set of primes $\mathcal P$ and $0\le z_p\le1$ we obtain
\[
 \E_{\QQ_a}\prod_{p\in\mathcal P}z_p^{V_{a,p}}
 =\prod_{p\in\mathcal P}
 \left(\frac{1-p^{-1-a}}{1-z_pp^{-1-a}}\right)^\theta.
\]
Thus the variables $(V_{a,p})_p$ are independent and
\[
 \QQ_a(V_{a,p}=k)
 =
 \frac{(\theta)_k}{k!}
 (1-p^{-1-a})^\theta p^{-k(1+a)},
 \qquad k\ge0.
\]
In particular, $V_{a,p}$ is negative binomial, and
\begin{equation}\label{eq:NB-mean}
 \E_{\QQ_a} V_{a,p}
 =
 \frac{\theta p^{-1-a}}{1-p^{-1-a}}
 \le 2\theta p^{-1-a}.
\end{equation}

Define the point process
\[
 \Xi_a=\sum_pV_{a,p}\delta_{a\log p}
\]
on $(0,\infty)$, where $\delta_u$ denotes the unit point mass at $u$,
and for a measurable test function $g$ write
$\Xi_a(g)=\int g(u)\,\Xi_a(du)$.  For fixed $a>0$, $\Xi_a$ is a finite
point measure almost surely, since
\[
 \sum_p\QQ_a(V_{a,p}>0)
 =\sum_p\bigl(1-(1-p^{-1-a})^\theta\bigr)<\infty
\]
and Borel--Cantelli applies.  The next lemma collects the prime number
theoretic estimates used below.

\begin{lemma}[Primes on the scale $a\log p$]\label{lem:mertens}
\textup{(i)}  As $a\downarrow0$, the measures $\sum_pp^{-1}\delta_{a\log p}$
converge vaguely on $(0,\infty)$ to $du/u$.  Consequently, for every
continuous function $h$ with compact support in $(0,\infty)$,
\begin{equation}\label{eq:prime-rescale}
 \sum_p p^{-1-a}h(a\log p)
 \longrightarrow
 \int_0^\infty h(u)\frac{e^{-u}}u\,du.
\end{equation}
\textup{(ii)}  There is a constant $C$ such that, for $0<a\le1$ and
$B>0$,
\begin{equation}\label{eq:prime-mass-bounds}
 a\sum_{a\log p\le B}\frac{\log p}{p}\le B+Ca,
 \qquad
 a\sum_{a\log p>B}\frac{\log p}{p^{1+a}}\le Ce^{-B}.
\end{equation}
\end{lemma}

\begin{proof}
(i) Mertens' second theorem \cite{MontgomeryVaughan2006}
says that, for a constant $c_0$,
$\sum_{p\le y}p^{-1}=\log\log y+c_0+o(1)$.  Hence the mass that
$\sum_pp^{-1}\delta_{a\log p}$ gives to an interval
$[r,s]\subset(0,\infty)$ is
\[
 \sum_{e^{r/a}\le p\le e^{s/a}}\frac1p
 =\log\frac{s}{a}-\log\frac{r}{a}+o(1)
 =\log\frac sr+o(1),
\]
which is the $du/u$ mass of $[r,s]$; since the limit is diffuse, this
gives vague convergence.  Multiplying by the compactly supported
continuous function $u\mapsto e^{-u}h(u)$ and writing
$p^{-1-a}=p^{-1}e^{-a\log p}$ proves \eqref{eq:prime-rescale}.

(ii) The first bound is Mertens' first theorem
\cite{MontgomeryVaughan2006},
$\sum_{p\le y}\log p/p=\log y+O(1)$, at $y=e^{B/a}$.  For the second, let
$\vartheta(y)=\sum_{p\le y}\log p$, so that $\vartheta(y)\le Cy$ by
Chebyshev's bound \cite{MontgomeryVaughan2006}, and put
$y_0=e^{B/a}$.  Stieltjes integration by parts
gives
\[
 \sum_{p>y_0}\frac{\log p}{p^{1+a}}
 =\int_{y_0^+}^\infty t^{-1-a}\,d\vartheta(t)
 \le(1+a)\int_{y_0}^\infty\vartheta(t)t^{-2-a}\,dt
 \le C(1+a)\frac{y_0^{-a}}{a},
\]
and $y_0^{-a}=e^{-B}$.
\end{proof}

\begin{proposition}[Rescaled logarithms of the prime factors]
\label{prop:point-process}
As $a\downarrow0$,
\begin{equation}\label{eq:Xi-convergence}
 \Xi_a\Longrightarrow\Xi
 \quad\text{vaguely on }(0,\infty),
\end{equation}
where $\Xi$ is a Poisson point process with intensity
\begin{equation}\label{eq:gamma-levy}
 \nu_\theta(du)=\theta\frac{e^{-u}}u\,du.
\end{equation}
If $L_{a,1}\ge L_{a,2}\ge\cdots$ are the atoms of $\Xi_a$,
listed with multiplicity and padded with zeros, and
$J_1\ge J_2\ge\cdots$ are those of $\Xi$, then
\begin{equation}\label{eq:l1-atoms}
 (L_{a,i})_{i\ge1}\Longrightarrow(J_i)_{i\ge1}
 \quad\text{in }\ell^1_\downarrow.
\end{equation}
\end{proposition}

\begin{proof}
Let $g:(0,\infty)\to[0,\infty)$ be continuous with compact support.
The negative binomial generating function gives
\begin{align*}
 -\log\E_{\QQ_a} e^{-\Xi_a(g)}
 &=
 \theta\sum_p\sum_{k\ge1}
 \frac{p^{-k(1+a)}}k
 \bigl(1-e^{-kg(a\log p)}\bigr).
\end{align*}
Applying \eqref{eq:prime-rescale} to $h=1-e^{-g}$ treats the $k=1$ term.
If $g$ vanishes outside $[\varepsilon,K]$, the remaining terms are at most
\[
 C\sum_{p\ge e^{\varepsilon/a}}p^{-2}=o(1).
\]
Consequently
\[
 \E_{\QQ_a} e^{-\Xi_a(g)}
 \longrightarrow
 \exp\left\{
 -\theta\int_0^\infty
 (1-e^{-g(u)})\frac{e^{-u}}u\,du
 \right\}.
\]
This is the Laplace functional of $\Xi$, so the Laplace functional criterion
\cite{Kallenberg2017} proves \eqref{eq:Xi-convergence}.

We next control the logarithmic mass near zero and near infinity.  By
\eqref{eq:NB-mean} and \eqref{eq:prime-mass-bounds}, for
$\varepsilon>0$, $K>0$ and $0<a\le1$,
\begin{align}
 \E_{\QQ_a}\int_{(0,\varepsilon]}u\,\Xi_a(du)
 &=\sum_{a\log p\le\varepsilon}a\log p\,\E_{\QQ_a} V_{a,p}
 \le2\theta a\sum_{a\log p\le\varepsilon}\frac{\log p}{p}
 \le2\theta(\varepsilon+Ca),
 \label{eq:small-mass-bound}\\
 \E_{\QQ_a}\int_{(K,\infty)}u\,\Xi_a(du)
 &\le2\theta a\sum_{a\log p>K}\frac{\log p}{p^{1+a}}
 \le C_\theta e^{-K}.
 \label{eq:large-mass-bound}
\end{align}
For the limit, \eqref{eq:gamma-levy} gives directly
$\E\int_{(0,\varepsilon]}u\,\Xi(du)=\theta(1-e^{-\varepsilon})\le\theta\varepsilon$
and $\E\int_{(K,\infty)}u\,\Xi(du)=\theta e^{-K}$.  Markov's inequality
therefore shows that the $\ell^1$ mass of the atoms below $\varepsilon$
and above $K$ becomes negligible in probability as, respectively,
$\varepsilon\downarrow0$ and $K\to\infty$, uniformly for small $a$ and
for the limit.

It remains to pass from point measures to ranked sequences.  Choose
levels $0<\varepsilon<K<\infty$; they are almost surely not atoms of
$\Xi$, because its intensity is diffuse.  Restriction to
$[\varepsilon,K]$, followed by listing the finitely many atoms in
decreasing order with multiplicity and padding with zeros, is continuous
in the vague topology at every point measure having no atom at
$\varepsilon$ or $K$.  Hence \eqref{eq:Xi-convergence} gives convergence
in distribution of these truncated ranked vectors in $\ell^1$.  The
$\ell^1$ distance between a full ranked vector and its truncation is
exactly the mass of the discarded atoms, which is uniformly negligible
in probability by \eqref{eq:small-mass-bound} and
\eqref{eq:large-mass-bound}, and the same holds for $\Xi$.  The standard
approximation theorem for weak convergence
\cite{Billingsley1999}, applied with
$(\varepsilon,K)=(1/j,j)$ and then $j\to\infty$, now proves
\eqref{eq:l1-atoms}.
\end{proof}

\begin{proposition}[Gamma and Poisson--Dirichlet limit under
$\QQ_a$]\label{prop:euler-limit}
Let $M_a\sim\QQ_a$.  As $a\downarrow0$,
\begin{equation}\label{eq:Gamma-PD}
 \left(a\log M_a,\W(M_a)\right)
 \Longrightarrow (G,\bP)
\end{equation}
in $\R_+\times\ell^1_\downarrow$, where
\[
 G\sim\Gam(\theta,1),
 \qquad
 \bP\sim\PD(\theta),
\]
and $G$ and $\bP$ are independent.
\end{proposition}

\begin{proof}
The atoms of $\Xi_a$ in decreasing order are precisely
\[
 (a\log P_i(M_a))_{i\ge1}.
\]
By Proposition~\ref{prop:point-process},
\[
 (a\log P_i(M_a))_{i\ge1}
 \Longrightarrow (J_i)_{i\ge1}
 \quad\text{in }\ell^1_\downarrow.
\]
The sum map is continuous on $\ell^1$, and therefore
\[
 a\log M_a
 =
 \sum_i a\log P_i(M_a)
 \Longrightarrow
 G:=\sum_iJ_i.
\]
The point process with intensity \eqref{eq:gamma-levy} is the jump
measure of the gamma subordinator described above on the time interval
$[0,\theta]$.  To identify the joint law, realize it as
the projection of a Poisson process $\sum_i\delta_{(T_i,J_i)}$ on
$[0,\theta]\times(0,\infty)$ with intensity
\[
 dt\,e^{-u}\frac{du}{u};
\]
its projection onto the second coordinate is $\Xi$.  The induced random
measure $\mathcal G(C)=\sum_{i:T_i\in C}J_i$ on $[0,\theta]$ has
independent masses on disjoint sets.  If $|C|$ denotes Lebesgue measure,
then for a measurable set $C$ and $s\ge0$,
\[
 \E\exp\{-s\mathcal G(C)\}
 =\exp\left\{-|C|\int_0^\infty(1-e^{-su})e^{-u}\frac{du}{u}\right\}
 =(1+s)^{-|C|}.
\]
If $|C|=0$, the intensity of points with time coordinate in $C$ is zero,
so $\mathcal G(C)=0$ almost surely.  If $|C|>0$, the preceding transform
identifies $\mathcal G(C)\sim\Gam(|C|,1)$.  Under this realization, the
subordinator above is $\mathcal G(t)=\mathcal G([0,t])$, and in particular
$G=\mathcal G([0,\theta])\sim\Gam(\theta,1)$.  On every finite measurable
partition, independent gamma variables with a common rate have a
Dirichlet vector of proportions which is independent of their sum.
Choose a nested countable family of finite partitions that generates the
Borel sigma field of $[0,\theta]$.  The factorization therefore holds for
every bounded cylinder function of the normalized masses.  Since these
cylinder functions generate the sigma field of the normalized random
measure, the monotone class theorem shows that
$\mathcal G/\mathcal G([0,\theta])$ is independent of its total.  The
ranked atom sizes are measurable functions of the normalized random
measure, and by the construction above they have law $\PD(\theta)$; see
Kingman \cite{Kingman1975} or Pitman \cite{Pitman2003}.
Consequently
\[
 G=\sum_iJ_i
 \quad\text{is independent of}\quad
 \left(\frac{J_i}{G}\right)_{i\ge1}.
\]
The normalization map is continuous at every nonzero sequence, and the
limiting sequence has norm $G>0$ almost surely.  At the prelimit zero
sequence we use the convention $\W(1)=(1,0,0,\ldots)$; its probability is
$\QQ_a(M_a=1)=\zeta(1+a)^{-\theta}\sim a^\theta$.  The extended
continuous mapping theorem applied to \eqref{eq:l1-atoms} now proves
\eqref{eq:Gamma-PD}, including the independence assertion.
\end{proof}

\section{Joint limits and the phase diagram}\label{sec:joint}

The first four subsections prove Theorems~\ref{thm:zero}
and~\ref{thm:joint}, one parameter regime at a time.  The final subsection
derives their common consequences and proves the corollaries, using the
cofactor results of Section~\ref{sec:arithmetic} and the Euler product
limit of Section~\ref{sec:euler}.

\subsection{The case \texorpdfstring{$\gamma=0$}{gamma = 0}}\label{sec:zero}

\begin{proof}[Proof of Theorem~\textup{\ref{thm:zero}}]
Let $\widehat N_x$ have the unperturbed cutoff law
\begin{equation}\label{eq:pure-cutoff}
 \widehat\Pp_x(\widehat N_x=n)
 =\frac{d_\theta(n)}{D_\theta(x)},
 \qquad n\le x,
\end{equation}
which is the cutoff law weighted by $d_\theta$ introduced in
Section~\ref{sec:billingsley}.
We write $\widehat\Pp_x$ and $\widehat\E_x$ for probability and
expectation under this law.  The laws in
\eqref{eq:model} and \eqref{eq:pure-cutoff} have the same conditional
distribution on $\{n\ge2\}$.  Hence
\begin{equation}\label{eq:zero-TV}
 d_{\TV}\bigl(\Law(N_x),\Law(\widehat N_x)\bigr)
 =
 \left|\frac1{Z_x}-\frac1{D_\theta(x)}\right|
 =O\bigl(D_\theta(x)^{-1}\bigr)=o(1).
\end{equation}

The generalized Billingsley theorem
\cite[Theorem~1.1]{ElboimGorodetsky2024} applies to the multiplicative
function $\alpha=d_\theta$.  In the notation of that theorem, take
$d=0$, $\mathfrak a=1/2$, $\eta=1/2$, and $r=3/2$.  Indeed,
$d_\theta(p)=\theta$ at every prime, so for some $c>0$ the prime number
theorem gives
\[
 \sum_{p\le y}d_\theta(p)\log p
 =\theta y+O\bigl(ye^{-c\sqrt{\log y}}\bigr)
 =\theta y+O\bigl(y(\log y)^{-1/2}\bigr).
\]
Moreover, $d_\theta(p)/p^d=\theta=O(p^{1/2-\eta})$ for
$\eta=1/2$, and, uniformly in the prime $p$,
\[
 d_\theta(p^k)
 =\frac{\Gamma(k+\theta)}{\Gamma(\theta)\Gamma(k+1)}
 =O_\theta\bigl((1+k)^{\max\{\theta-1,0\}}\bigr)
 =O_\theta\bigl((3/2)^k\bigr),
 \qquad k\ge1.
\]
Thus the hypotheses of their theorem hold, and it yields
\begin{equation}\label{eq:EG-PD}
 \left(
 \frac{\log P_i(\widehat N_x)}{\log x}
 \right)_{i\ge1}
 \Longrightarrow\bP,
 \qquad \bP\sim\PD(\theta),
\end{equation}
initially in the product topology.  Regular variation of $D_\theta$ gives
\begin{equation}\label{eq:pure-location}
 \widehat\Pp_x(\widehat N_x/x\le u)
 =\frac{D_\theta(ux)}{D_\theta(x)}\longrightarrow u,
 \qquad 0<u\le1,
\end{equation}
where integer parts are understood.  It also gives
\[
 \frac{\log\widehat N_x}{\log x}\xrightarrow{\widehat{\Pp}}1,
\]
since, for every $\varepsilon>0$,
\[
 \widehat\Pp_x(\widehat N_x\le x^{1-\varepsilon})
 =\frac{D_\theta(x^{1-\varepsilon})}{D_\theta(x)}=o(1).
\]

It remains to upgrade \eqref{eq:EG-PD} to the $\ell^1$ topology.  Put
\[
 \mathbf X_x=
 \left(\frac{\log P_i(\widehat N_x)}{\log x}\right)_{i\ge1},
 \qquad
 T_x=\frac{\log\widehat N_x}{\log x},
\]
so that $\|\mathbf X_x\|_1=T_x$.  By \eqref{eq:EG-PD}, the convergence
$T_x\xrightarrow{\widehat{\Pp}}1$, and Slutsky's theorem,
$(\mathbf X_x,T_x)\Rightarrow(\bP,1)$, where the sequence
coordinate carries the product topology.  On a Skorohod representation,
the coordinates and their sums converge almost surely.  For nonnegative
sequences this implies $\ell^1$ convergence, by first truncating at a
fixed coordinate.  Hence $\mathbf X_x\Rightarrow\bP$ in
$\ell^1_\downarrow$.  On $\{\widehat N_x\ge2\}$, we have
$\W(\widehat N_x)=\mathbf X_x/T_x$.  Since $T_x\to1$ in
probability and
\[
 \widehat\Pp_x(\widehat N_x=1)=D_\theta(x)^{-1}=o(1),
\]
the extended continuous mapping theorem gives
\begin{equation}\label{eq:pure-PD-logN}
 \W(\widehat N_x)\Longrightarrow\bP
 \quad\text{in }\ell^1_\downarrow.
\end{equation}

It remains to prove independence.  If $f$ is bounded and continuous
on $\ell^1_\downarrow$, $u\in(0,1]$, and $y=\lfloor ux\rfloor$,
let $\widehat N_y$ denote the same cutoff variable under $\widehat\Pp_y$.
The cutoff family has the exact consistency identity
\begin{align*}
 \widehat\E_x\left[
 f(\W(\widehat N_x))
 \one_{\{\widehat N_x/x\le u\}}
 \right]
 =
 \frac{D_\theta(y)}{D_\theta(x)}
 \widehat\E_y f(\W(\widehat N_y)).
\end{align*}
The right side tends to
$u\E f(\bP)$ by \eqref{eq:pure-location} and
\eqref{eq:pure-PD-logN}.  Since $\widehat N_x/x$ takes values in
$[0,1]$ and $\W(\widehat N_x)$ is tight, these mixed expectations
determine the joint limit, and their factorized form proves independence
in \eqref{eq:zero-joint}.  Equation~\eqref{eq:zero-TV} transfers the
result to $N_x$.
\end{proof}

\subsection{The case \texorpdfstring{$0<\gamma<1$}{0 < gamma < 1}}

\begin{proof}[Proof of Theorem~\textup{\ref{thm:joint}\textup{(i)}}]
Put $L=\log x$.  Set
$\tau_x(m)=(a_x\log m,\W(m))$.  Since total variation contracts under
measurable maps, Theorem~\ref{thm:cofactor} gives
\[
 d_{\TV}\bigl(\Law(\tau_x(R_x)),\Law(\tau_x(M_{a_x}))\bigr)
 \longrightarrow0.
\]
Proposition~\ref{prop:euler-limit} therefore yields
\begin{equation}\label{eq:cofactor-joint}
 (a_x\log R_x,\W(R_x))
 \Longrightarrow(G,\bP).
\end{equation}

It remains to prove joint convergence with $N_x/x$.  Let $\Phi$ be bounded and
continuous on $\R_+\times\ell^1_\downarrow$.  We shall prove that, for
every $u\in(0,1]$,
\begin{align}
 &\E_x\left[
 \Phi(a_x\log R_x,\W(R_x))
 \one_{\{N_x/x\le u\}}
 \right]                                                   \notag\\
 &\hspace{18mm}
 -
 u\E_{\QQ_{a_x}}
 \Phi(a_x\log M_{a_x},\W(M_{a_x}))
 \longrightarrow0.                                      \label{eq:uniform-factorization}
\end{align}
For $u=1$, the indicator is identically one and the assertion follows
directly from Theorem~\ref{thm:cofactor}.  We may therefore assume
$0<u<1$.
It is enough to work on $\{P^+(N_x)>\sqrt x\}$, whose complement has
probability tending to zero.  Given the cofactor $m<\sqrt x$, put
$y=x/m$.  The unnormalized weight of the primes for which
$pm/x\le u$ is $\bigl(A(uy)-A(\sqrt x)\bigr)_+$.  On the effective set
\[
 \mathcal T_x=\{m:a_x\log m\le4\log L\},
\]
we have $\log m=o(L)$, so $uy>\sqrt x$ for all sufficiently
large $x$.  Since $\log y\ge L/2\to\infty$,
equation~\eqref{eq:prime-ratio-u} gives, uniformly on
$\mathcal T_x$, the ratio $A(uy)/A(y)\to u$.  The estimate
\eqref{eq:A-sqrt-negligible} also gives
$A(\sqrt x)/A(y)\to0$ uniformly on this set.  Therefore
\begin{equation}\label{eq:A-u-ratio}
 \frac{A(uy)-A(\sqrt x)}
 {A(y)-A(\sqrt x)}
 \longrightarrow u
 \qquad\text{uniformly for }m\in\mathcal T_x.
\end{equation}
Put $\Phi_x(m)=\Phi(a_x\log m,\W(m))$.  The exact contribution from
$\{P^+(N_x)>\sqrt x\}$ to the first expectation in
\eqref{eq:uniform-factorization} is
\begin{equation}\label{eq:uniform-exact-sum}
 \frac{\theta}{Z_x}
 \sum_{m<\sqrt x}d_\theta(m)\Phi_x(m)
 \bigl(A(ux/m)-A(\sqrt x)\bigr)_+.
\end{equation}
Define, with value zero when the denominator vanishes,
\[
 \rho_{x,u}(m)=
 \frac{\bigl(A(ux/m)-A(\sqrt x)\bigr)_+}
 {A(x/m)-A(\sqrt x)}.
\]
On the support of the conditional cofactor law $\mu_x^>$, the
denominator is positive and $0\le\rho_{x,u}\le1$.  Hence
\eqref{eq:uniform-exact-sum} is exactly
\[
 \frac{Z_x^>}{Z_x}
 \E_{\mu_x^>}\bigl[\Phi_x(R_x)\rho_{x,u}(R_x)\bigr].
\]
The complement of $\mathcal T_x$ has vanishing mass under
$\mu_x^>$ and $\QQ_{a_x}$, by
Proposition~\ref{prop:cofactor-comparison} and
the bound \eqref{eq:Qa-exp-moment} in
Lemma~\ref{lem:Qa-exp-moment}.  Moreover,
\[
 \left|\E_{\mu_x^>}\left[
 \Phi_x(R_x)\{\rho_{x,u}(R_x)-u\}\right]\right|
 \le \|\Phi\|_\infty
 \left(
  \sup_{\substack{m\in\mathcal T_x\\ \mu_x^>(m)>0}}
  |\rho_{x,u}(m)-u|
  +\mu_x^>(\mathcal T_x^c)
 \right)=o(1)
\]
by \eqref{eq:A-u-ratio}.  Finally,
\[
 \left|\E_{\mu_x^>}\Phi_x(R_x)
 -\E_{\QQ_{a_x}}\Phi_x(M_{a_x})\right|
 \le2\|\Phi\|_\infty d_{\TV}(\mu_x^>,\QQ_{a_x})=o(1),
\]
while $Z_x^>/Z_x\to1$ and the omitted event
$\{P^+(N_x)\le\sqrt x\}$ has probability tending to zero.  Combining
these estimates proves \eqref{eq:uniform-factorization}.  By
Proposition~\ref{prop:euler-limit}, the comparison term converges to
$u\E\Phi(G,\bP)$.  Since $N_x/x$ takes values in $[0,1]$ and the
cofactor coordinates are tight by \eqref{eq:cofactor-joint}, the
functions $(v,z)\mapsto\one_{\{v\le u\}}\Phi(z)$, with $0<u<1$,
form a determining class.
Thus \eqref{eq:uniform-factorization} proves \eqref{eq:joint-limit} and
the mutual independence of the three limits.
\end{proof}

\subsection{The case \texorpdfstring{$\gamma=1$}{gamma = 1}}

\begin{proof}[Proof of Theorem~\textup{\ref{thm:joint}\textup{(ii)}}]
Let $h:\N\to\R$ be bounded.  We first prove that, for every
$u\in(0,1]$,
\begin{equation}\label{eq:critical-factorization}
 \E_x\left[h(R_x)\one_{\{N_x/x\le u\}}\right]
 \longrightarrow
 u^{1+\beta}\E_{\QQ_\beta}h(M_\beta).
\end{equation}
For $u=1$, the indicator is identically one and the assertion follows
from Theorem~\ref{thm:cofactor}.  We may therefore assume $0<u<1$.
The complement of $\{P^+(N_x)>\sqrt x\}$ has probability tending to
zero.  On $\{P^+(N_x)>\sqrt x\}$, the unnormalized numerator on the left of
\eqref{eq:critical-factorization} is
\[
 \theta\sum_{m<\sqrt x}d_\theta(m)h(m)
 \bigl(A(ux/m)-A(\sqrt x)\bigr)_+.
\]
Define
\[
 \rho_{x,u}(m)=
 \frac{\bigl(A(ux/m)-A(\sqrt x)\bigr)_+}
 {A(x/m)-A(\sqrt x)},
\]
with value zero if the denominator vanishes.  Thus
\[
 \E_x\left[h(R_x)\one_{\{N_x/x\le u\}}\right]
 =\frac{Z_x^>}{Z_x}\,
 \E_{\mu_x^>}\left[h(R_x)\rho_{x,u}(R_x)\right]+o(1).
\]
Fix $M$.  Uniformly for $m\le M$,
\eqref{eq:prime-ratio-u} and
$A(\sqrt x)/A(x/m)\to0$ give
\[
 \frac{A(ux/m)-A(\sqrt x)}
 {A(x/m)-A(\sqrt x)}\longrightarrow u^{1+\beta}.
\]
The ratio $\rho_{x,u}$ lies in $[0,1]$, so the
normalized contribution of the terms $m>M$ is at most
$\|h\|_\infty\mu_x^>(\{m\in\N:m>M\})$.
Proposition~\ref{prop:cofactor-comparison}
gives
$d_{\TV}(\mu_x^>,\QQ_\beta)\to0$, and $\QQ_\beta$ is a fixed law.
Since $Z_x^>/Z_x\to1$, truncating, letting $x\to\infty$, and then
$M\to\infty$
proves \eqref{eq:critical-factorization}.  The convergence for all
$u$ and bounded $h$ determines the joint law, and its factorized form
proves \eqref{eq:critical-joint}, including independence.
\end{proof}

\subsection{The case \texorpdfstring{$\gamma>1$}{gamma > 1}}

\begin{proof}[Proof of Theorem~\textup{\ref{thm:joint}\textup{(iii)}}]
By \eqref{eq:super-Z-prime}, it is enough to condition on $N_x$ being
prime.  Under that conditioning its law is
\[
 \Pp_x(N_x=p\mid N_x\text{ is prime})=\frac{e^{\beta(\log p)^\gamma}}{A(x)}.
\]
For fixed $t\ge0$, \eqref{eq:prime-boundary-ratio} with $y=x$ gives
\begin{align*}
 \Pp_x\left(b_x\log\frac{x}{N_x}\ge t
 \mathrel{\big|}N_x\text{ is prime}\right)
 =\frac{A(xe^{-t/b_x})}{A(x)}
 \longrightarrow e^{-t}.
\end{align*}
Under the conditioning, $R_x=1$.  Removing the conditioning changes the
joint law in total variation by at most
$\Pp_x(N_x\text{ is not prime})=o(1)$, which proves
\eqref{eq:super-boundary}.  Since $b_x\to\infty$, its tightness also
implies $N_x/x\to1$ in probability.
\end{proof}

\subsection{Consequences}

\begin{proof}[Proof of Corollary~\textup{\ref{cor:giant-prime}}]
First suppose $0<\gamma\le1$.  Theorem~\ref{thm:joint} gives
$N_x/x\Rightarrow U$ for $0<\gamma<1$ and
$N_x/x\Rightarrow B_\beta$ for $\gamma=1$.  Both limiting variables are
positive almost surely.  For every $\eta,\delta>0$, once
$x^{-\eta}<\delta$,
\[
 \Pp_x\left(\frac{\log(N_x/x)}{\log x}\le-\eta\right)
 \le \Pp_x(N_x/x\le\delta).
\]
Taking the upper limit and then letting $\delta\downarrow0$ gives
\begin{equation}\label{eq:log-N}
 \frac{\log N_x}{\log x}\xrightarrow{\Pp}1.
\end{equation}
If $0<\gamma<1$, then $a_x\log R_x$ is tight and
$a_x\log x=\beta\gamma(\log x)^\gamma\to\infty$.  If $\gamma=1$,
then $R_x\Rightarrow M_\beta$ and is tight.  Consequently, in both
cases,
\begin{equation}\label{eq:log-R-small}
 \frac{\log R_x}{\log x}\xrightarrow{\Pp}0.
\end{equation}
Also, $\Pp_x(N_x=1)=Z_x^{-1}\to0$ by
Theorem~\ref{thm:partition}.  On $\{N_x\ge2\}$,
\[
 \frac{\log P^+(N_x)}{\log N_x}
 =1-\frac{\log R_x}{\log N_x},
 \qquad
 \|\W(N_x)-(1,0,0,\ldots)\|_1
 =2\frac{\log R_x}{\log N_x}.
\]
Together with \eqref{eq:log-N} and \eqref{eq:log-R-small}, these
identities prove \eqref{eq:giant} and \eqref{eq:full-degenerate} for
$0<\gamma\le1$.

If $\gamma>1$, \eqref{eq:super-Z-prime} shows that $N_x$ is prime with
probability tending to one.  On that event $P^+(N_x)=N_x$ and
$\W(N_x)=(1,0,\ldots)$, which proves both assertions in the remaining
case.
\end{proof}

\begin{proof}[Proof of Corollary~\textup{\ref{cor:defect}}]
In every regime, $\Pp_x(N_x=1)=Z_x^{-1}\to0$ by
Theorem~\ref{thm:partition}, so it is enough to work on
$\{N_x\ge2\}$, where $\mathcal D_x=\log R_x/\log N_x$.

\smallskip
\noindent\emph{The defect limits.}
The case $\gamma=0$ follows from Theorem~\ref{thm:zero}, since the first
coordinate of $\W(N_x)$ is $1-\mathcal D_x$.  If
$0<\gamma<1$, then
\[
 (\log x)^\gamma \mathcal D_x
 =\frac{a_x\log R_x}{\beta\gamma}
 \frac{\log x}{\log N_x},
\]
and the conclusion follows from \eqref{eq:joint-limit} and
\eqref{eq:log-N}.  At $\gamma=1$, use
$(\log x)\mathcal D_x=(\log x/\log N_x)\log R_x$ and
\eqref{eq:critical-joint}.  For $\gamma>1$,
\eqref{eq:super-Z-prime} gives $\Pp_x(\mathcal D_x=0)\to1$.

\smallskip
\noindent\emph{The ordinary scale.}
For $\gamma=0$, Theorem~\ref{thm:zero} gives
$\mathcal D_x\Rightarrow1-S_1$, where $1-S_1>0$ almost surely, while
$\log N_x\to\infty$ in probability.  Since
$\mathcal D_x=\log R_x/\log N_x$, for $M>1$ and $\varepsilon>0$,
\[
 \Pp_x(R_x\le M)
 \le
 \Pp_x(\mathcal D_x\le\varepsilon)
 +\Pp_x\left(\log N_x\le\frac{\log M}{\varepsilon}\right).
\]
Letting first $x\to\infty$ and then $\varepsilon\downarrow0$ shows
that $R_x\to\infty$ in probability.  For $0<\gamma<1$ and every
fixed $M>1$,
\[
 \Pp_x(R_x\le M)
 =\Pp_x(a_x\log R_x\le a_x\log M)\longrightarrow0,
\]
because $a_x\downarrow0$ and $G$ has no atom at zero.
At $\gamma=1$, $P^+(N_x)/N_x=R_x^{-1}$ and
\eqref{eq:critical-joint} gives
$P^+(N_x)/N_x\Rightarrow M_\beta^{-1}$.  If $\gamma>1$, the same ratio
equals one whenever $N_x$ is prime, whose probability tends to one by
\eqref{eq:super-Z-prime}.  This proves the remaining lines of
\eqref{eq:ordinary-ratio-regimes}.
\end{proof}

\begin{proof}[Proof of Corollary~\textup{\ref{cor:prime-freezing}}]
For $\gamma=0$, the total Gibbs weight of the prime integers is
$e^\beta\theta\pi(x)$, where $\pi(x)=\#\{p\le x\}$.  The prime number
theorem and \eqref{eq:Z-zero}
give
\[
 \Pp_x(N_x\text{ is prime})
 \sim\frac{\Gamma(\theta+1)}{(\log x)^\theta}.
\]
For $\gamma>0$, the total Gibbs weight of the prime integers is exactly
$\theta A(x)$.  Combining Lemma~\ref{lem:weighted-pnt} with
\eqref{eq:Z-unified} gives, for $0<\gamma\le1$,
\[
 \Pp_x(N_x\text{ is prime})
 =\frac{\theta A(x)}{Z_x}
 \sim\zeta(1+a_x)^{-\theta}.
\]
This is asymptotic to $a_x^\theta$ when $0<\gamma<1$, and equals
$\zeta(1+\beta)^{-\theta}$ when $\gamma=1$.
The case $\gamma>1$ is \eqref{eq:super-Z-prime}.
\end{proof}

\section*{Funding}
This work was supported by the National Natural Science Foundation of China
(grant no.~12401171) and by the National Key R\&D Program of China
(grant no.~2022YFA1006500).

\end{document}